\documentclass{article}

\usepackage[top=3cm,bottom=3.5cm,inner=3cm,outer=3cm]{geometry}
\usepackage{amsthm,amsmath,amssymb,mathtools}
\usepackage{enumerate,csquotes,verbatim}
\usepackage{setspace}
\usepackage[hidelinks]{hyperref}
\usepackage{framed}
\usepackage{cleveref}
\usepackage{thm-restate}
\usepackage{bbm}
\usepackage{color}
\usepackage{subcaption}

\usepackage[square,sort,comma,numbers]{natbib}
\MakeOuterQuote{"}

\makeatletter
\def\imod#1{\allowbreak\mkern11mu({\operator@font mod}\,\,#1)}

\makeatother

\renewenvironment{proof}[1][]{\begin{trivlist}
\item[\hspace{\labelsep}{\bf\noindent Proof#1.\/}] }{\qed\end{trivlist}}

\newenvironment{proof-claim}[1][]{\begin{trivlist}
\item[\hspace{\labelsep}{\it\noindent Proof#1.\/}] }{\qed\end{trivlist}}

\theoremstyle{plain}
\newtheorem{theorem}{Theorem}
\Crefname{theorem}{Theorem}{Theorems}

\newtheorem{conjecture}[theorem]{Conjecture}
\Crefname{conjecture}{Conjecture}{Conjectures}

\newtheorem{lemma}[theorem]{Lemma}
\Crefname{lemma}{Lemma}{Lemmas}

\newtheorem{claim}[theorem]{Claim}
\Crefname{claim}{Claim}{Claims}

\Crefname{definition}{Definition}{Definitions}

\newtheorem{proposition}[theorem]{Proposition}
\Crefname{proposition}{Proposition}{Propositions}

\Crefname{observation}{Observation}{Observations}

\Crefname{corollary}{Corollary}{Corollaries}

\Crefname{remark}{Remark}{Remarks}

\newtheorem{example}[theorem]{Example}
\Crefname{example}{Example}{Examples}

\Crefname{construction}{Construction}{Constructions}

\newcommand{\ceil}[1]{\lceil #1 \rceil}
\newcommand{\one}{\mathbbm{1}}

\DeclareMathOperator*{\PP}{\mathbb{P}}
\DeclareMathOperator*{\R}{\mathcal{R}}
\DeclareMathOperator*{\B}{\mathcal{B}}

\newcommand{\comp}[1]{\overline{#1}}

\begin{document}
\title{Partitioning a graph into monochromatic connected subgraphs}
\author{Ant\'onio Gir\~ao\thanks{
		Department of Pure Mathematics and Mathematical Statistics, 
		University of Cambridge, Cambridge, UK;
		e-mail: \texttt{A.Girao}@\texttt{dpmms.cam.ac.uk}.
	}
	\and Shoham Letzter\thanks{
        ETH Institute for Theoretical Studies,
        ETH,
        8092 Zurich;
        e-mail: \texttt{shoham.letzter}@\texttt{eth-its.ethz.ch}.
    }
	\and Julian Sahasrabudhe\thanks{
		Department of Mathematics, 
		University of Memphis, 
		Memphis, 
		Tennessee; 
		e-mail: \texttt{julian.sahasra}@\texttt{gmail.com}.
	}}
\maketitle

\begin{abstract}
    \setlength{\parindent}{0in} 
    \setlength{\parskip}{.08in} 
    \noindent
	A well-known result by Haxell and Kohayakawa states that the vertices of
	an $r$-coloured complete graph can be partitioned into $r$ monochromatic
	connected subgraphs of distinct colours; this is a slightly weaker
	variant of a conjecture by Erd\H{o}s, Pyber and Gy\'arf\'as that states
	that there exists a partition into $r-1$ monochromatic connected
	subgraphs.
	We consider a variant of this problem, where the complete graph is
	replaced by a graph with large minimum degree, and prove two conjectures
	of Bal and DeBiasio, for two and three colours. 

    \setlength{\parskip}{.1in} 
\end{abstract}

\section{Introduction} \label{sec:introduction}
	An old observation by Erd\H{o}s and Rado says that when the edges of a 
	complete graph are coloured with two colours, there is a spanning
	monochromatic component. This simple remark has been the starting point
	of extensive research. A natural example is the search for large
	monochromatic components in $r$-edge-coloured complete graph (see, for
	example, \cite{gyarfas,gyarfas-survey}).
	Here we focus on a different direction, namely, the search for covers (or
	partitions) of the vertices into as few as possible monochromatic
	connected subgraphs.
	
	A classical example appears in a seminal paper by Erd\H os, Gy\' arf\' as
	and Pyber \cite{ErdosGyarfasPyber}, who showed that for any $r$-colouring
	of $K_n$ (the complete graph on $n$ vertices) the vertices can be
	partitioned into at most $O(r^{2}\log r)$ monochromatic cycles.
	We note that throughout the paper, when we say that the vertices of a
	graph are covered (or partitioned) by a collection of subgraphs, we
	mean that the vertices are covered by the \emph{vertex sets} of these subgraphs.
	
	Gy\'arf\'as, Ruszink\'o, S\'ark\"ozy and Szemer\'edi \cite{GRSS} improved
	the above result by showing that if the edges of the complete graph are
	$r$-coloured then the vertices can be partitioned into $O(r \log r)$
	monochromatic cycles. In the other direction, Pokrovskiy
	\cite{pokrovskiy} showed that one needs strictly more than $r$ cycles,
	disproving a conjecture of Erd\H{o}s, Gy\' arf\' as and Pyber
	\cite{ErdosGyarfasPyber}. Conlon and Stein \cite{conlon-stein} showed
	similar results for colourings where every vertex is incident with at
	most $r$ distinct colours.  The question of whether one can partition an
	$r$-coloured graph into $O(r)$ monochromatic cycles remains an enticing
	open problem in this area. 
	
	In a slightly different direction, Erd\H{o}s, Gy\' arf\' as and Pyber
	\cite{ErdosGyarfasPyber} conjectured that the vertices of an $r$-coloured
	complete graph may be partitioned into at most $r-1$ monochromatic
	connected subgraphs. 
	This conjecture
	is known to be tight when $r-1$ is a prime power and $n$ is sufficiently
	large, due to a well-known construction which requires the existence of
	an affine plane of an appropriate order. Haxell and Kohayakawa
	\cite{haxell-kohayakawa} proved a slightly weaker result, showing that
	one can partition an $r$-coloured complete graph on $n$ vertices into $r$
	monochromatic subgraphs, for sufficiently large $n$.

	Interestingly, this problem is closely related to a well-known conjecture
	of Ryser on packing and covering edges in $r$-partite, $r$-uniform
	hypergraphs. This link was first noted by Gy\'{a}rf\'{a}s \cite{gyarfas}
	in 1997 and leads to the following natural formulation of the conjecture
	of Ryser, published in \citep{ryser}, where
	$\alpha(G)$ is the size of the largest independent set in the graph $G$.	
	\begin{conjecture}[Ryser (see \cite{ryser})]
		\label{conj:ryser-lovasz}
		The vertex set of an $r$-coloured graph $G$ can be covered by at most
		$(r-1)\alpha(G)$ monochromatic connected subgraphs.
	\end{conjecture}

	In this form, it is clear that Ryser's conjecture implies the
	\textit{covering version} of the aforementioned conjecture of Erd\H{o}s,
	Gy\'arf\'as and Pyber about monochromatic connected subgraphs. 
	Although not much is known
	about Ryser's conjecture in general, a few special cases are understood. The
	case $r=2$ is equivalent to  K\"onig's classical theorem (see
	\cite{Diestel}, for example), while the case $r=3$ was
	proved by Aharoni \cite{aharoni} in 2001, who built on the earlier advances of
	Aharoni and Haxell \cite{aharoni-haxell}. The conjecture is also known to
	hold 
	for $\alpha(G) = 1$ (i.e.\ $G$ is a complete graph) and $r \le 5$, as was
	proved by Gy\'arf\'as \cite{gyarfas} ($r = 3$), Duchet \cite{duchet} and
	Tuza \cite{tuza} ($r = 4$), and Tuza \cite{tuza} ($r = 5$).
	
	Following Schelp \cite{schelp} who suggested several variants of
	Ramsey-type problems (e.g.\ the determining the length of the longest
	monochromatic path in a $2$-coloured graphs), we consider variants of the
	above problems for graphs with large minimum degree.  Our first main
	result proves a conjecture of Bal and DeBiaso \cite{bal-debiasio} about
	partitioning the vertices of a $2$-coloured graph with large minimum
	degree; recall that $\delta(G)$ denotes the minimum degree of the graph
	$G$. 
	
	\begin{restatable}{theorem}{thmTwoColsPartn} \label{thm:two-cols-partn}
		There exists an integer $n_0$ such that that every $2$-coloured graph
		$G$ on $n \geq n_0$ vertices and with minimum degree at least
		$\frac{2n-5}{3}$ can be partitioned into two monochromatic connected
		subgraphs. 	
	\end{restatable}
	
	We note that this is a generalisation of the result by Haxell and
	Kohayakawa \cite{haxell-kohayakawa} mentioned above for two colours,
	where the complete graph is replaced by a graph with large minimum
	degree.  This result is seen to be sharp by a construction of Bal and
	DeBiasio \cite{bal-debiasio}; in \Cref{sec:conclusion} we describe a more
	general family of examples which shows, in particular, the sharpness of
	the minimum degree condition in this result. One can think of this result
	as saying that $\frac{2n-5}{3}$ is the minimum degree `threshold' that
	guarantees a partition of every 2-colouring into two monochromatic
	connected subgraphs. It is therefore natural to ask what minimum degree
	condition on a graph $G$ guarantees a partition into $t$ monochromatic
	connected subgraphs, no matter how the graph is $2$-coloured. We conjecture the
	following.

	\begin{restatable}{conjecture}{conjPartitioningTwoColours} 
		\label{conj:partitioning-two-colours}
		For every $t$ there exists $n_0$, such that for every
		$2$-colouring of a graph $G$ on $n\geq n_0$ vertices with
		$\delta(G)\geq \frac{2n-2t-1}{t+1}$ there exists a partition of
		the vertex set into at most $t$ monochromatic connected subgraphs.
	\end{restatable}

	We support this conjecture by observing an analogous result for
	\emph{covers} of the vertices by monochromatic components. 

	\begin{proposition}\label{prop:two-colouring}
		Let $t$ be integer and let $G$ be a $2$-coloured graph on $n$
		vertices with $\delta(G) \geq \frac{2n-2t-1}{t+1}$. Then vertices of
		$G$ can be covered by at most $t$ monochromatic components.
	\end{proposition}	
	
	We also give a construction, showing that the inequality in this
	proposition (and therefore the conjecture) cannot be improved.

	Bal and DeBiasio \cite{bal-debiasio} also considered the problem of
	covering coloured graphs with monochromatic components of distinct
	colours. In particular, they conjectured the following.
	\begin{restatable}{conjecture}{conjBalDebiasioDcover}
		\label{conj:bal-debiasio-dcover}
			Let $G$ be an $r$-coloured graph on $n$ vertices with $\delta(G)
			\ge (1-1/2^r)n$. Then the vertices can be covered by
			monochromatic components of distinct colours.
	\end{restatable}
	
	Again, Bal and DeBiasio provided examples showing that if true, the bound
	$(1-2^{-r})n$ is best possible. We shall prove
	\Cref{conj:bal-debiasio-dcover} for $r = 2,3$. The case $r = 3$ is the
	most interesting case but we include a short proof of $r = 2$ for
	completeness. 
	
	\begin{restatable}{theorem}{thmCoverDistinctThree}
		\label{thm:cover-distinct-three}
		Let $G$ be a $3$-coloured graph on $n$ vertices with $\delta(G) \ge
		7n/8$. Then the vertices of $G$ can be covered by monochromatic
		components of distinct colours.
	\end{restatable}
	
	\subsection{Structure of the paper}

		We conclude the introduction with a description of the notation that
		we shall use in this paper.
		We prove
		\Cref{thm:two-cols-partn} in \Cref{sec:partition-two-colours}, and
		prove \Cref{thm:cover-distinct-three} in
		\Cref{sec:cover-distinct-colours}. We conclude the paper in
		\Cref{sec:conclusion} with some final remarks and open problems and
		give a proof of \Cref{prop:two-colouring}.

	\subsection{Notation}\label{subsec:NotAndCons}

		By an \emph{$r$-coloured} graph, we mean a graph whose edges are
		coloured with $r$ colours.  When a graph is $2$-coloured we call the
		colours \emph{red} and \emph{blue}; and when it is $3$-coloured, we
		denote the colours by \emph{red}, \emph{blue} and \emph{yellow}.  
		
		For a set of vertices $W$, we denote by $N_r(W)$ the set of vertices
		in $V(G)\setminus W$ that are adjacent to a vertex in $W$ by a red
		edge.  If $x \in V(G)$ is a vertex, we define $d_r(x) = |N_r(\{x\})|$
		which we refer to as the \emph{red degree} of $x$.  We say that $y$
		is a \emph{red neighbour} of $x$ if $xy$ is a red edge.  By a
		\emph{red component} of a graph $G$, we mean the \emph{vertex set} of
		a component in the graph on vertex set $V(G)$ whose edgse are the
		red edges of $G$. We denote the red component that contains $x$ by
		$C_r(x)$.  A \emph{red set} $U$ is a set of vertices that is
		connected in red, i.e.\ the red edges induced by $U$ form a connected
		graph.

		All the above definitions and notation, that were defined for red,
		also works for blue or yellow; e.g.\ $d_b(x)$ and $d_y(x)$ are the
		blue and yellow degrees of $x$, respectively, and a blue set is a set
		of vertices that is connected in blue.

\section{Partitioning into monochromatic connected subgraphs}
	\label{sec:partition-two-colours}
	In this section we prove \Cref{thm:two-cols-partn}. 
	\thmTwoColsPartn*	
	We note that the minimum degree condition in this theorem cannot be
	improved; this can be seen by taking $t = 2$ in \Cref{ex:cover-t}, which
	we describe in \Cref{sec:conclusion}.

	\begin{proof}[ of \Cref{thm:two-cols-partn}] 
		Throughout this proof, we assume that the number of vertices
		$n$ is sufficiently large.  
		Suppose, for a contradiction, that
		the vertices of $G$ cannot be partitioned into two monochromatic sets.

		\begin{claim}\label{claim:large-blue-compt}
			There is a blue component of order at most $(n+1)/6$. 
		\end{claim}
		\begin{proof-claim}[ of \Cref{claim:large-blue-compt}]
			We may assume that there are at least three red
			components and at least three blue components, as otherwise the
			vertices may be partitioned into two red sets or two blue sets
			(recall that a \emph{red set} is defined to be a set of vertices
			that is connnected in red, and similarly for blue),
			contradicting our assumption.  Let $R$ be a red
			component of smallest order; so, $|R| \leq n/3$. 
			
			Let us assume first that $|R|\leq (n-5)/3$. Since every vertex in
			$R$ sends at least $(2n-5)/3 - (|R| - 1) > (n - |R|)/2$ blue
			edges outside of $R$, every two vertices in $R$ have a common blue
			neighbour outside of $R$. Hence, $R$ is contained in a blue
			component of order at least $|R| + (2n-5)/3 - (|R| - 1)
			\ge (2n-1)/3$. Since there are at least three blue components,
			there is a blue component of order at most $(n-(2n-1)/3)/2=
			(n+1)/6$. 
			
			We now assume that $(n-4)/3 \le |R| \leq n/3$. If every two
			vertices in $R$ have a common blue neighbour, then, again, $R$ is
			contained in a blue component of order at least $(2n-1)/3$ and as
			before there is a blue component of order at most $(n+1)/6$.
			Otherwise, there exist two vertices $u, w\in R$ whose blue
			neighbourhoods do not intersect. But both $u$ and $v$ have at
			least $(n-5)/3$ blue neighbours outside of $R$, and therefore
			every vertex in $R \setminus \{u,v\}$ has a common blue neighbour
			with either $u$ or $w$. It follows that there are two blue
			components (namely, the components $C_b(u)$ and $C_b(w)$) whose
			union has order at least $|R| + 2(n-5)/3 > n - 5$, hence there is
			a blue component of order at most $4$.
		\end{proof-claim}
		
		\begin{claim} \label{claim:red-star}
			There is a red set $U$ of size at most $27\log n$
			such that $|N_{r}(U)| \geq 2n/3-27\log n$.	
		\end{claim}

		\begin{proof-claim}[ of \Cref{claim:red-star}]
			By the previous claim, there is a blue component $B$ of
			order at most $(n+1)/6$. Note that
			every vertex in $B$ has at
			least $(2n-5)/3-|B|$ red neighbours in $V(G)\setminus B$. Fix a
			vertex $u \in B$ and let $N$ be the set of red neighbours of $u$
			outside $B$. Every $w \in B$ has at least the following number of
			red neighbours in $N$.
			\begin{align*}
				2 \cdot ((2n-5)/3-|B|)-(n-|B|) 
				  =  (n-10)/3- |B|  
				\geq (n-21)/6.
			\end{align*} 
			Now let $U'$ be a random subset of $N$ where
			each vertex $w \in N$ belongs to $U'$, independently, with probability
			$13 \log n/n$. Let $I_w$ be the event that $w$ (where $w \in B$)
			does not have a red neighbour in $U'$. We bound
			\begin{equation*}
				\mathbb{P} \Big(\bigcup_{w \in B} I_w\Big) \leq  
				|B| \cdot \mathbb{P}(I_w) \leq 
				n \cdot \left(1-\frac{13\log n}{n}\right)^{\frac{n-21}{6}}\leq 
				n \cdot e^{-2\log n} < 1/2. 
			\end{equation*}
			Note that since $\mathbb{E}(|U'|) \le 13 \log n$, we have
			$\mathbb{P}(|U'|\geq 26\log n) \le 1/2$, by Markov's inequality.
			Therefore, there is a
			choice of $U' \subseteq N$ such that $|U'| \leq 26\log n$ and
			every vertex in $B$ is joined by a red edge to some
			vertex in $U'$.  We choose $U = U' \cup \{u\}$. Note that
			\begin{align*} 
				N_r(U' \cup \{u\})| 
				& \geq |N \setminus U'| + |B \setminus \{u\}| \\
				& \geq \left( (2n-5)/3 - |B| - 26\log n \right) + (|B|-1) \\ 
				& = 2n/3 - 27\log n .
			\end{align*} 
			Hence, the set $U = U' \cup \{u\}$ satisfies the requirements of
			\Cref{claim:red-star}.
		\end{proof-claim}
			
		Let $U$ be a red set as in \Cref{claim:red-star} and
		let $N = N_r(U)$. Now choose a maximal sequence of distinct vertices
		$x_1,\ldots,x_t \in V \setminus (N \cup U)$ so that $x_i$ has at
		least $\log n$ red neighbours in the set $N \cup \{x_1,\ldots,x_{i-1}
		\}$, for every $i \in [t]$. Then put $\comp{N} = N \cup \{x_1, \ldots,
		x_t\}$ and write $W
		= V(G) \setminus \left(U \cup \comp{N}\right)$. Note that every
		vertex in $W$
		has at most $\log n$ red neighbours in $\comp{N}$.
		\begin{claim} \label{claim:comp-N-small}
			$\left|\comp{N}\right| \le 2n/3 + 3 \log n + 4$.  
		\end{claim}
		\begin{proof-claim}[ of \Cref{claim:comp-N-small}] 
			For a contradiction, suppose that $\left|\comp{N}\right| > 2n/3 +
			3 \log n + 4$.  We shall deduce that the vertices can be
			partitioned into a red set and a blue one, a contradiction.

			To define the partition, fix $w \in W$ and let $X = N_b(w) \cap
			\comp{N}$.  Let $S$ be a random subset of $X$, obtained by taking
			each vertex of $X$ independently with probability $1/2$. We claim
			that, with positive probability, $(U \cup \comp{N}) \setminus S$
			is red \emph{and} $W \cup S$ is blue. 
			
			To bound the probability that $W \cup S$ is blue, we consider the
			probability that every vertex in $W$ is joined by a blue edge to
			$S$ (an event which would imply that $W \cup S$ is blue).  For
			every $x,y \in V$ we have $|N(x)\cap N(y)| \geq n/3 - 10/3$,
			hence $|N(x) \cap N(y) \cap \comp{N}| \geq 3\log n$.  Since every
			vertex in $W$ has at most $\log n$ red neighbours in $\comp{N}$,
			we have $|N_b(x) \cap N_b(y) \cap \comp{N} | \geq \log n $.
			Therefore the probability that a given $x \in W$ has no blue
			neighbours in $S$ is at most $2^{-\log n} = 1/n$. Thus, the
			expected number of vertices in $W$ with no edges to $S$ is
			smaller than $1/2$ ( note that $|W|\leq n/3$). Hence, $\PP( W
			\cup S \textit{ is blue}) > 1/2$. 
			
			We now estimate the probability that $(U \cup \comp{N}) \setminus
			S$ is red. First note that as $N = N_r(U)$, we have that $U \cup
			N'$ is red for any subset $N' \subseteq N$.  So it remains to
			show that the vertices of $\{x_1,\ldots,x_t\} \setminus S$ can be
			joined, via a red path, to $U \cup (N \setminus S)$, with
			sufficiently high probability. For $i \in [t]$, let $E_i$ be the
			event that vertex $x_i$ is joined by a red edge to $(N \cup
			\{x_1,\ldots,x_{i-1}\})\setminus S$.  Note that if the event $E =
			\bigcap_i^t E_i$ holds, $(U \cup \comp{N})\setminus S$ is red.
			Now, to estimate $\PP(E_i)$, for $i \in [t]$, note that each
			vertex $x_i$ has at least $\log n$ forward neighbours, and the
			probability that one of these vertices is deleted is at most
			$1/2$. Thus $\PP(E_i) \ge 1 - 2^{-\log n} = 1 - 1/n$, therefore
			$\PP\!\left((U \cup \comp{N}) \setminus S \text{ is red}\right) \ge
			\PP(E) > 1/2$, where the second inequality holds since $t < n/2$.
			
			Thus, with positive probability, $W \cup S$ is blue and $(U \cup
			\comp{N}) \setminus S$ is red. In particular, the vertices can be
			partitioned into a blue set and a red one, a contradiction.
		\end{proof-claim}

		\begin{claim} \label{claim:bigRedVertex}
			There is a vertex of blue degree at most $90 \log n$. 
		\end{claim}
		\begin{proof-claim}[ of \Cref{claim:bigRedVertex}]
			By definition of $\comp{N}$ and since $\left|\comp{N}\right| \ge
			2n/3 - 27\log n$, every vertex in $W$ has at least $n/3 - 29 \log
			n $ blue neighbours in $\comp{N}$. 
			
			Fix a vertex $w \in W$. If there is a vertex $v \in W$ with
			$|N_b(v) \cap N_b(w) \cap \comp{N}| < \log n$, then the two blue
			components containing $v$ and $w$ cover all vertices of $W$ and
			all but at most $62 \log n$ vertices of $\comp{N}$ (as
			$\left|\comp{N}\right| \le 2n/3 + 3\log n + 4$, by the previous
			claim). Since $|U| \le 27 \log n$, it follows that these two
			components cover all but at most $90 \log n$ vertices. Recall
			that there are at least three blue components, hence there is a
			component of order at most $90 \log n$, and any vertex in that
			component has blue degree at most $90 \log n$. 

			Otherwise, every vertex $v\in W$ satisfies $|N_b(v) \cap N_b(w)
			\cap \comp{N}| \ge \log n$. As in \Cref{claim:comp-N-small}, let
			$S$ be an uniformly random subset of $N_b(w) \cap \comp{N}$; we find that,
			with positive probability, $\left(U \cup \comp{N}\right)
			\setminus S$ is red and $W \cup S$ is blue, so the vertices can
			be partitioned into a red set and a blue one, a contradiction to
			our assumption.
		\end{proof-claim}

		Let $r$ be a vertex of blue degree at most $90 \log n$, which exists
		by the previous claim. By symmetry, there is a vertex $b$ of red
		degree at most $90 \log n$. Then $d_r(r), d_b(b) \ge 2n/3 - 90\log n
		- 2$. Write $A_1 = N_b(b) \setminus N_r(r)$, $A_2 = N_b(b) \cap
		N_r(r)$ and $A_3 = N_r(r) \setminus N_b(b)$.  Then $|A_2| \ge n/3 -
		180\log n - 4$ and $|A_1|, |A_3| \le n/3 + 90\log n + 2$.

		\begin{claim} \label{claim:third-vx}
			There is a vertex with no blue neighbours in $A_1$, no red
			neighbours in $A_3$, and at most $2\log n$ neighbours in $A_2$.
		\end{claim}

		\begin{proof-claim}[ of \Cref{claim:third-vx}]
			Suppose that the statement does not hold. Let $\{B, R\}$ be a
			random partition of $A_2$, obtained by putting vertices in $B$,
			independently, with probability $1/2$. Then, with positive
			probability, every vertex in $G$ has a blue neighbour in $A_1
			\cup B \subseteq  N_b(b)$ or a red neighbour in $A_3 \cup R
			\subseteq N_r(r)$. We thus obtain a
			partition of the vertices into a red set and a blue one, a
			contradiction. 		
		\end{proof-claim}

		Let $x$ be a vertex with no blue neighbours in $A_1$, no red
		neighbours in $A_3$, and at most $2\log n$ neighbours in $A_2$ (its
		existence is guaranteed by the previous claim). Then
		$|A_2| \le n/3 + 3\log n$, so $|A_1|, |A_3| \ge n/3 - 95\log n$.
		Furthermore, $x$ has at least $n/3 - 100\log n$ red neighbours in
		$A_1$ and at least $n/3 - 100 \log n$ blue neighbours in $A_3$.
		Write $A_1' = A_1 \cap N_r(x)$, $A_2' = A_2 \setminus N(x)$, and
		$A_3' = A_3 \cap N_b(x)$ (so $|A_1'|, |A_3'| \ge n/3 - 100\log n$ and
		$|A_2'| \ge n/3 - 190\log n$).

		\begin{claim} \label{claim:distinct-compts}
			The vertices $x$ and $b$ are in distinct blue components;
			similarly, $x$ and $r$ are in distinct red components.
		\end{claim}

		\begin{proof-claim}[ of \Cref{claim:distinct-compts}]
			Suppose that $x$ and $b$ are in the same blue component.  Then
			there is a blue path $P$ from $\{x\} \cup A_3'$ to $\{b\} \cup
			A_1' \cup A_2'$. We may assume that the inner vertices of $P$ are
			outside of $A_1' \cup A_2' \cup A_3' \cup \{x, b\}$. Hence, $|P|
			\le 400 \log n$.

			Now, let $\{B, R\}$ be a random partition of $(A_2' \cup A_3')
			\setminus V(P)$, obtained by putting each vertex in $B$,
			independently, with probability $1/2$. It is easy to see that,
			with positive probability, every vertex in $G$ has a red
			neighbour in $R$ or a blue neighbour in $B$, from which it can be
			deduced that there is a partition of the vertices into a red set
			and a blue one, which is a contradiction. Indeed, note that
			$P\cup\{x,b\}\cup B$ is a blue set and $\{r\}\cup R$ is a red
			set. Thus, we have that $b$ and $x$ are in distinct blue
			components; by symmetry, $r$ and $x$ are in different red
			components.
		\end{proof-claim}

		Note that $|C_b(b)|,|C_r(r)| \ge 2n/3 - 91 \log n$ and
		$|C_b(x)|,|C_r(x)| \ge n/3 - 100\log n$.  Recall that there are at
		least three blue components.  Hence, there is a vertex $r_1$ which is
		not in $C_b(b)$ or in $C_b(x)$. It follows that $d_{b}(r_1)$ is at
		most $191\log n$, hence it has red degree at least $2n/3 - 192 \log
		n$, so $r_1 \in C_r(r)$. Similarly, there is a vertex $b_1$ which is
		not in $C_r(r)$ or in $C_r(x)$, and therefore it must belong to
		$C_b(b)$.  We claim that the set $X=\{b_1,r_1,x\}$ is independent. We
		cannot have $r_1\tilde x$ or $b_1\tilde x$, for this either
		contradicts the choice of $r_1\notin C_{b}(x)$ and $b_1 \notin
		C_{r}(x)$ or it contradicts the statement of
		\Cref{claim:distinct-compts}. If we had $r_1 \tilde b_1$ and this
		edge was coloured red then $b_1 \in C_{r}(r)$ which is a
		contradiction, by definition of $b_1$. If $r_1b_1$ is coloured blue
		then we arrive at the contradiction $r_1 \in C_{b}(b)$.  Thus $X$ is
		independent. So finally, by the minimum degree condition, there must
		be a vertex $w$ that is adjacent to all three vertices in $X$.
		Indeed, if no such $w$ exists, then the number of edges between $X$
		and $V(G) \setminus X$ is at most $2(n-3) < 3(2n-5)/3$, a
		contradiction.  Without loss of generality, $w$ sends two red edges
		into $X$, implying that two of these vertices in $X$ belong to the
		same red component, a contradiction. This completes our proof of
		\Cref{thm:two-cols-partn}.
	\end{proof}

\section{Covering with monochromatic components of distinct colours}
\label{sec:cover-distinct-colours}

	In this section we verify \Cref{conj:bal-debiasio-dcover} for $r \in
	\{2,3\}$. Most of the difficulty is in the proof for $r=3$, but we
	include a short proof for $r=2$ for completeness.
	Actually, the $r=2$ case (for $n$ large) already follows from a difficult 
	result of Letzter \cite{letzter}, who showed that when $\delta(G) \ge
	3n/4$, the vertices can be partitioned into
	two monochromatic \emph{cycles} of different colours, for every
	$2$-colouring of $G$. Before turning to the 
	proofs, we mention the following construction of Bal and DeBiasio
	\citep{bal-debiasio}, which shows that the minimum degree condition in
	\Cref{conj:bal-debiasio-dcover} cannot be improved.

	\begin{example}
		Let $n \geq 2^r$; we shall define a graph on vertex set $[n]$ as
		follows.
		Partition $[n]$, as equally as possible, into $2^r$ sets which are
		indexed by the sequences $s \in \{0,1\}^r$. We write 
		\[ [n] = \bigcup_{s \in \{0,1\}^r} A(s) \] 
		and define the following, where $\one = (1, \ldots, 1)$.
		\[E = [n]^{(2)} \setminus \bigcup_{s \in \{0,1\}^r} \{ xy : x \in
		A(s), y \in  A(\one - s) \}.\] 
		In other words, we include all edges in the graph except for the edges
		between parts of the partition corresponding to antipodal elements of
		$\{0,1\}^r$.  Now, colour all edges $xy$, where $x \in A(s),y \in
		A(s')$, by the first coordinate on which $s,s'$ agree; e.g.\ the edge
		between $(0,1,0,0)$ and $(1,0,0,1)$ is coloured $3$. 
	
		We now show that $G$ cannot be covered by components of distinct
		colours. Suppose that it can, and note that the $i$-coloured
		components are of the form $\bigcup_{s \in S_i} A(s)$ 
		where $S_i$ is a set of elements that agree on their $i$-th
		coordinate; denote this coordinate by $a_i$. 
		It follows that the vertices of
		$A((1-a_1,\ldots,1-a_r))$ are not covered by any of these components,
		a contradiction. 
	\end{example}
	
	We now prove \Cref{conj:bal-debiasio-dcover} for $r = 2$.
	\begin{lemma} 
		Let $G$ be a $2$-coloured graph with $\delta(G) \ge 3n/4$. Then
		the vertices of $G$ can be covered by a red component and a blue
		component.
	\end{lemma}

	\begin{proof}
		We first show that there is a monochromatic component of order
		greater than $n/2$. If $G$ is red connected we are done. Hence, there
		exists a red component $R$ with $|R| \leq n/2$. Then, any two
		vertices $u, w \in R$ have a common blue neighbour, as $|N_b(u) \cap
		N_b(w) \cap R| \geq 2\cdot(3n/4-(|R|-1)) - (n - |R|) > 0$. So $R
		\subseteq C_b(u)$ and $C_b(u)$ is a blue component of order at least
		$3n/4$, as required. 

		Without loss of generality, there is a red component $R$ of order
		larger than $n/2$.  Note that there is a vertex $x$ which is not in
		$R$ (otherwise we are done), and $|N_b(x)\cap R| = |N(x)\cap R| >
		n/4$, as $x$ does not send red edges to $R$. In particular,
		$|C_b(x)\cap R| > n/4$. It follows that every vertex sends at least
		one edge to $C_b(x) \cap R$ and thus the components $R$ and $C_b(x)$
		cover the whole graph.
	\end{proof}

	We now turn to prove \Cref{thm:cover-distinct-three}, which is the case of
	three colours in \Cref{conj:bal-debiasio-dcover}.
	\thmCoverDistinctThree*

	\begin{proof} We begin with a series of preparatory claims
		(\Cref{claim:three-intersection,claim:two-intersection,%
		claim:one-intersection}).
		
	\begin{claim} \label{claim:three-intersection}
		If there are three monochromatic components of distinct colours whose
		intersection has order at least $n/8$, then the vertices can be
		covered by monochromatic components of distinct colours. 
	\end{claim}

	\begin{proof-claim}[ of \Cref{claim:three-intersection}]
		Suppose that $R$, $B$ and $Y$ are red, blue and yellow components, whose
		intersection $U =  R \cap B \cap Y$ has size at least
		$n/8$. Then, by the minimum degree condition, every vertex not in
		$U$ has a neighbour in $U$, implying that every vertex in the
		graph belongs to at least one of $R$, $B$ and $Y$, as required. 
	\end{proof-claim}

	\begin{claim} \label{claim:two-intersection}
		If there are two monochromatic components of distinct colours
		whose intersection has order at least $n/4$, then the vertices of $G$
		may be covered by monochromatic components of distinct colours.
	\end{claim}

	\begin{proof-claim}[ of \Cref{claim:two-intersection}]
		Suppose that $R$ and $B$ are red and blue components whose
		intersection $U = R \cap B$ has size at least $n/4$.
		We show that one of the following holds.
		
		\begin{enumerate}
			\item
				$R \cup B = V(G)$; 
			\item
				there is a yellow component whose intersection 
				with $R \cap B$ has size at least $n/8$. 
		\end{enumerate}
		
		Suppose that the first assertion does not hold. Then there is a
		vertex $u \notin R \cup B$. By the minimum degree condition, $u$
		sends at least $n/8$ edges to $R \cap B$, but these edges cannot be
		red or blue (because $u \notin R \cup B$), hence they are yellow, so
		by picking $Y$ to be the yellow component containing $u$, the second
		assertion holds.  If the first assertion holds, we are done
		immediately; otherwise, we are done by
		\Cref{claim:three-intersection}. 
	\end{proof-claim}

	\begin{claim} \label{claim:one-intersection}
		If there is a monochromatic component of order at least $n/2$, then
		the vertices can be covered by three monochromatic components of
		distinct colours.
	\end{claim}

	\begin{proof-claim}[ of \Cref{claim:one-intersection}]
		As in the proof of \Cref{claim:two-intersection}, we show that
		one of the following assertions holds, where $R$ is a red
		component of order at least $n/2$.
		\begin{enumerate}
			\item 
				$R = V(G)$;
			\item 
				there are monochromatic components $B$ and $Y$ 
				in colours blue and yellow respectively, such that 
				$R \cup B \cup Y = V(G)$;
			\item 
				there are monochromatic components $B$ and $Y$ 
				in colours blue and yellow respectively, such that 
				$|R \cap B \cap Y| \ge n/8$.
		\end{enumerate}
		Suppose that $R \neq V(G)$ and let $u \notin R$. Consider the
		blue and yellow components, $B$ and $Y$, containing $u$. 
		By the minimum degree
		condition, $u$ sends at least $|R| - n/8$ edges to $R$, none
		of which are red. So
		$|(B\cup Y)\cap R|\geq |R|-n/8$. Suppose that $R$, $B$ and
		$Y$ do not cover the whole graph. Let
		$w \notin R \cup B \cup Y$, and denote the blue and yellow
		components containing $w$ by $B'$ and $Y'$.
		By the same argument
		as before, $|(B' \cap Y') \cap R|\geq |R|-n/8$, which
		implies the following. 
		\begin{equation*}
			\left|(B \cup Y) \cap (B' \cup Y') \cap R\right| \geq 
			|R| - n/4 \ge n/4.	
		\end{equation*}
		Since $B \cap B' =\emptyset$ and $Y \cap Y' = \emptyset$, 
		either $|B \cap Y' \cap R| \ge n/8$ or $|B'
		\cap Y \cap R| \ge n/8$. This completes the proof that one of
		the above assertions holds. If one of the first two assertions holds,
		we are done immediately; and if the third assertion holds,
		\Cref{claim:one-intersection} follows
		from \Cref{claim:three-intersection}.
	\end{proof-claim}

	Henceforth, we assume $G$ cannot be covered by monochromatic
	components of distinct colours.

	\begin{claim} \label{claim:large-compt}
		There are two monochromatic components of distinct colours of order
		at least $3n/8$.
	\end{claim}

	\begin{proof-claim}[ of \Cref{claim:large-compt}]
		We will show that for every pair of colours there is a
		monochromatic component of order at least $3n/8$ in one of the two
		colours; the claim easily follows from this fact. Let the
		two colours be red and blue.  Since $G$ is not connected in yellow,
		we may find a partition $\{X, Y\}$ of the vertices of $G$ such that no $X -
		Y$ edges are yellow.  Without loss of generality, at least half the
		edges between $X$ and $Y$ are red; set $H = G_r[X, Y]$, denote $d(x)
		= d_H(x)$ for any vertex $x$, and given an
		edge $xy$ in $H$, set $s(xy) = d(x) + d(y)$.  We will show that there
		is an edge $xy$ with $s(xy) \ge 3n/8$; note that this would imply the
		existence of a red component of order at least $3n/8$, as required.
		Put $e = e(H)$ and, without loss of generality, we assume that $|X|
		\le |Y|$. We have
		\begingroup
		\addtolength{\jot}{.5em} 
		\begin{align*}
			\frac{1}{e}\sum_{xy \in E(H)}s(xy) = \,\, & 
			\frac{1}{e}\sum_{xy \in E(H)}(d(x) + d(y)) \\
			= \,\, & \frac{1}{e} \left( \sum_{x \in X}d(x)^2 + \sum_{y \in
			Y}d(y)^2\right) \\
			\ge \,\, & \frac{1}{e}\left(\frac{\big(\sum_{x \in
			X}d(x)\big)^2}{|X|} 
			+ \frac{\big(\sum_{y \in Y}d(y)\big)^2}{|Y|}\right) \\
			= \,\, & e \left(\frac{1}{|X|} + \frac{1}{|Y|}\right) \\
			\ge \,\, & \frac{1}{2}|X|\left(|Y| - n/8\right)\left(\frac{1}{|X|} +
			\frac{1}{|Y|}\right) \\
			= \,\, & \frac{1}{2}\left(|Y| - n/8 + |X| - \frac{|X| \cdot n/8}{n -
			|X|}\right) \\
			\ge \,\,& 3n/8.
		\end{align*}
		\endgroup
		Indeed, the first inequality follows from the Cauchy-Schwarz
		inequality; the second follows from the minimum degree condition and
		the assumption that red is the majority colour between $X$ and $Y$;
		and the last inequality follows since $|X| + |Y| = n$ and the
		expression $\frac{|X|}{n - |X|}$ is maximised when $|X| = n/2$ (as we
		have the constraint $|X| \le n/2$).
		
		This chain of inequalities shows that the average value of $s(xy)$ is
		at least $3n/8$; in particular, there is a red component of order at
		least $3n/8$, as required. 
	\end{proof-claim}
		
	We remark that the idea of double counting $s(xy)$ as in the proof of the
	previous claim originated in a paper by Liu, Morris and Prince
	\cite{liu-morris-prince}.
		
	By the previous claim, we may assume that $R$ and $B$ are red and blue
	components of order at least $3n/8$.

	\begin{claim} \label{claim:difference} 
		Either $|R \setminus B| < n/4$ or $|R \setminus B| < n/4$.
	\end{claim}
	\begin{proof-claim}[ of \Cref{claim:difference}]
		Assume that $|R \setminus B| \geq n/4$ and $|R \setminus B| \geq
		n/4$. Note that every edge between the disjoint sets $R \setminus B$
		and $B \setminus R$ is yellow. Furthermore, any two vertices in $B
		\setminus R$ has a common neighbour in $R \setminus B$, and vice
		versa.
		Therefore $B \triangle R$ in contained in a yellow component; in
		particular, there is ia yellow components of order at least $n/2$,
		a contradiction, by \Cref{claim:one-intersection}.
	\end{proof-claim}
	
	By the previous claim, we may assume that $|B \setminus R| < n/4$.
	Hence, $|B \cup R|  = |R| + |B \setminus R| < n/2 + n/4 = 3n/2$, by
	\Cref{claim:one-intersection}. Therefore, the set $W = V(G) \setminus (R
	\cup B)$ has size larger than $n/4$. Since all edges between $R \cap B$
	and $W$ are yellow, it follows that every two vertices in $R \cap B$ have
	a common yellow neighbour in $W$ and hence $R \cap B$ is contained in a
	yellow component.  Thus, \Cref{claim:three-intersection} implies that
	$|R \cap B| < n/8$. It follows that $|B| = |B \cap R| + |B \setminus R| <
	3n/8$, in contradiction with the choice of $B$. This completes the proof
	of \Cref{thm:cover-distinct-three}. 
\end{proof}

\section{Concluding remarks} \label{sec:conclusion}

	As possible lines for future research, we remind the reader of
	\Cref{conj:bal-debiasio-dcover} by Bal and DeBiasio \cite{bal-debiasio};
	in this paper we proved this conjecture for $r \le 3$.
	\conjBalDebiasioDcover*
	
	Another conjecture stated by Bal and Debiasio \citep{bal-debiasio}
	concerns the minimum degree needed to ensure that an $r$-coloured
	graph can be covered by at most $r$ monochromatic components, whose
	colours need not to be distinct.
	\begin{conjecture}
		Let $G$ be an $r$-coloured graph on $n$ vertices with $\delta(G) \ge
		\frac{r(n-r-1)+1}{r+1}$. Then the vertices of $G$ can be covered by
		at most $r$ monochromatic components.
	\end{conjecture}

	We further recall our \Cref{conj:partitioning-two-colours}.
	\conjPartitioningTwoColours*
	
	In this conjecture, we attempt to determine the minimum degree condition
	needed to guarantee the existence of a partition of a $2$-coloured graph
	into $t$ monochromatic connected subgraphs. This is a generalisation of
	\Cref{thm:two-cols-partn} which determines this condition for a partition
	into two monochromatic connected sets.
	To close the section, we  prove \Cref{prop:two-colouring}, a weaker
	version of \Cref{conj:partitioning-two-colours}, where instead of
	partitioning the vertices, we aim to \emph{cover} the vertices.

	\begin{proof-claim}[ of \Cref{prop:two-colouring}]
		We use the link with K\"onig's Theorem
		first noted by Gy\'{a}rf\'{a}s \citep{gyarfas}. 
		Let $G$ be a $2$-coloured
		graph with minimum degree at least $\frac{2n-2t-1}{t+1}$.
		Let $\R$ be the collection of red components (some of which may be
		singletons, if there are vertices that are not incident with any red
		edges), and let $\B$ be the collection of blue edges.
		Define an auxiliary bipartite graph $H = (\R, \B, E)$, where for $R
		\in \R$ and $B \in \B$, we have $RB \in E$ if and only if $R
		\cap B \neq \emptyset$.
	
		We claim that there is no matching of size larger than $t$. Indeed,
		suppose that $\{R_1 B_1, \ldots, R_{t+1} B_{t+1}\}$ is a matching
		of size $t + 1$. Let $u_i \in R_i \cap B_i$, for $i \in [t+1]$
		and $U = \{u_1, \ldots, u_{t+1}\}$. Then the vertices of $U$ are in
		distinct red and blue components. In particular, $U$ is independent,
		so the number of edges between $U$ and $V(G) \setminus U$ is at least
	 $2n-2t-1$.  On the other hand, no vertex sends more than one red edge
		into $U$ (and similarly for blue), so every vertex not in $U$ sends
		at most two edges into $U$.  It follows that the number of edges
		between $U$ and $V(G) \setminus U$ is at most $2(n - t - 1) <  2n -2t
		-1$, a contradiction.

		By K\"onig's theorem, which states that in bipartite graphs,
		the size of a minimum cover equals the
		size of a maximum matching, it follows that there is a cover $W$ of
		size at most $t$; write
		$W = \{C_1, \ldots, C_t\}$. We claim that $V(G) = C_1 \cup
		\ldots \cup C_t$. Indeed, consider a vertex $u$ and denote its red
		and blue components by $R$ and $B$, respectively. Then $R \cap B \neq
		\emptyset$, hence $RB$ is an edge in $H$, so either $R$ or $B$ is in
		$W$, which implies that $u \in C_1 \cup \ldots \cup C_t$, as
		required. In other words, the vertices of $G$ can be covered by at
		most $t$ monochromatic components.
	\end{proof-claim}

	Finally, we note that the restriction on the minimum degree in
	\Cref{prop:two-colouring} (and therefore
	\Cref{conj:partitioning-two-colours}) cannot be improved.
	The special case of this example, where $t = 2$, appears in
	\cite{bal-debiasio} and shows that the minimum degree condition in
	\Cref{thm:two-cols-partn} is best possible.

	\begin{example} \label{ex:cover-t}
		Let $U$ be a set of size $n \ge t+1$, and let $\{X, A_1, \ldots,
		A_{t+1}\}$ be a partition of $U$, where $|X| = t + 1$ and the sizes
		of $A_1, A_2,\ldots, A_{t+1}$ are as equal as possible; write $X= \{
		x_1, \ldots, x_{t+1} \}$.
		We define a $2$-coloured graph $G$ on vertex set $U$ as follows.
		\begin{itemize}
			\item
				the sets $A_i$ are cliques, and we colour them arbitrarily;
			\item
				we add all possible edges between $A_i$ and $A_{i + 1}$,
				where $i \in [t]$, and
				colour them red if $i$ is odd, and blue otherwise;
			\item
				we add all edges between $x_i$ and $A_i \cup A_{i + 1}$, for
				$i \in [t+1]$ (addition is taken modulo $t + 1$).
				We colour these edges red if $i$ is in $[t]$ and $i$ is odd;
				and blue if $i$ is in $[t]$ and $i$ is even.
				Finally, we colour the edges from $x_{t+1}$ to $A_1$ blue,
				and colour the edges from $x_{t+1}$ to $A_{t+1}$ red if
				$t$ is even and blue if $t$ is odd.
		\end{itemize}
		\begin{figure}[h]
			\centering
			\includegraphics[]{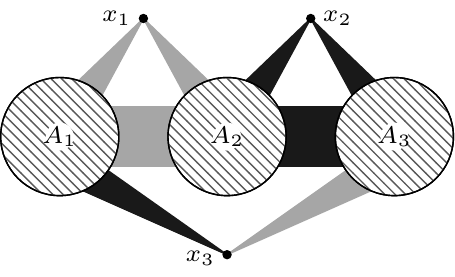}
			\hspace{1cm}
			\includegraphics[]{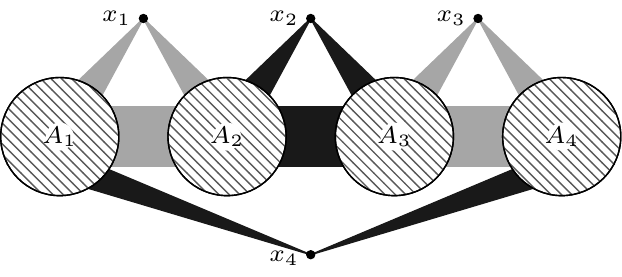}
			\caption{an illustration of \Cref{ex:cover-t} for $t = 2$ and $t
		= 3$ (here grey represents red and black represents blue).}
			\label{fig:ex-t-three}
		\end{figure}
		An easy calculation shows that $G$ has minimum
		degree\footnote{
			In fact, we need to be a bit more careful here. 
			Write $n = a(t +
			1) + r$, where $a$ and $r$ are integers and $0 \le r \le t$.
			We consider two cases: $r < \ceil{(t + 1)/2}$ and $r \ge \ceil{(t
			+ 1)/2}$. In the former case, it is easy to see that $\delta(G) =
			\ceil{(2n - 2t - 1)/(t + 1)} - 1$. In the latter case, note that
			exactly $r$ of the sets $A_i$ have size $a$, and the rest have
			size $a - 1$. Then, again, one can check that $\delta(G) =
			 \ceil{(2n -2t - 1)/(t + 1)} - 1$ if $|A_i| = a$ for every odd $i
			\in [t+1]$ (which is possible as $r \ge (t + 1)/2$).
		}
		$\ceil{(2n -2t - 1)/(t + 1)} - 1$, and that no two vertices in $X$
		belong to the same monochromatic component; in particular, the
		vertices of $G$ cannot be covered by at most $t$ monochromatic
		components.
	\end{example}

\subsection*{Acknowledgements}
	We would like to thank B\' ela Bollob\' as for many valuable comments.
	We would also like to thank Louis DeBiasio for bringing \cite{tuza} to
	our attention.
	The second author would like to acknowledge the support of Dr.~Max
	R\"ossler, the Walter Haefner Foundation and the ETH Zurich Foundation.
	The third author would like to thank Trinity College, Cambridge for
	support.

\bibliographystyle{amsplain}
\bibliography{monocomponent}

\end{document}